\documentclass[11pt]{article}

\usepackage[a4paper,margin=1in]{geometry}
\usepackage{amsmath,amssymb,amsthm,mathtools}
\usepackage[utf8]{inputenc}

\input{amssym.def}
\input{amssym}

\newcommand{\pa}{\partial}
\newcommand{\hl}{m}
\newcommand{\un}{\underline}
\newcommand{\qq}{q^{-1}}
\DeclareMathOperator{\Tr}{Tr}
\newcommand{\MM}{\mathcal{M}}
\newcommand{\de}{\delta}
\newcommand{\De}{\Delta}
\newcommand{\ot}{\otimes}
\newcommand{\C}{\mathbb{C}}
\DeclareMathOperator{\Sym}{Sym}
\renewcommand{\AA}{\mathcal{A}}
\newcommand{\ov}{\overline}
\newcommand{\al}{\alpha}
\DeclareMathOperator{\Cas}{Cas}
\newcommand{\DD}{\mathcal{D}}
\newcommand{\DDDD}{\mathfrak{D}}
\newcommand{\UU}{U(gl(N)_h)}
\newcommand{\UUU}{U(u(2)_h)}
\newcommand{\tpa}{\widetilde{\pa}}
\DeclareMathOperator{\End}{End}
\newcommand{\hh}{\hbar}
\newcommand{\hatt}{\widehat{\Theta}}
\newcommand{\tatt}{\widetilde{\Theta}}
\DeclareMathOperator{\Mat}{Mat}
\newcommand{\rh}{r_{\hbar}}
\newcommand{\Ah}{\mathcal{B}}
\newcommand{\AAA}{\mathcal{A}}
\newcommand{\parh}{\pa_{\rh}}
\newcommand{\Om}{\Omega}
\newcommand{\SSS}{\mathcal{S}}
\newcommand{\Aqh}{\mathcal{M}_{h}(R)}
\newcommand{\om}{\omega}
\newcommand{\tQ}{\widetilde Q}
\newcommand{\dd}{\mathbf{d}}
\DeclareMathOperator{\Ch}{Ch}
\newcommand{\Dif}{\mathsf{Dif}}

\numberwithin{equation}{section}

\theoremstyle{plain}
\newtheorem{theorem}{Theorem}[section]

\newtheorem{lemma}[theorem]{Lemma}
\newtheorem{proposition}[theorem]{Proposition}
\theoremstyle{definition}

\newtheorem{remark}[theorem]{Remark}

\begin{document}

\title{Characteristic classes on a noncommutative background: a new approach}

\author{\rule{0pt}{7mm} Dimitry Gurevich\thanks{gurevich@ihes.fr}\\
	{\small\it Higher School of Modern Mathematics, Moscow Institute of Physics and Technology,}\\
	{\small\it 9 Institutskiy Pereulok, 141700 Dolgoprudny, Moscow Region, Russian Federation}\\
	\rule{0pt}{7mm} Vladimir Roubtsov\thanks{volodya@univ-angers.fr}\\
	{\small\it 	LAREMA, UMR 6093 du CNRS, D\'epartement de  Math\'ematiques,}\\
	{\small\it Universit\'e d'Angers, 2, boulevard Lavoisier, Angers, France}\\
	\rule{0pt}{7mm} Pavel Saponov\thanks{Pavel.Saponov@ihep.ru}\\
	{\small\it 	National Research University Higher School of Economics,}\\
	{\small\it 20 Myasnitskaya Ulitsa, Moscow 101000, Russian Federation}\\
	{\small\it and}\\
	{\small\it Institute for High Energy Physics, NRC "Kurchatov Institute"}\\
	{\small \it Protvino 142281, Russian Federation}}

\maketitle


\smallskip


\bigskip

\begin{abstract}
	We introduce analogues of partial derivatives with respect to the generators of the enveloping algebras $U(gl(N))$ and their
	$q$-deformations, the so-called modified reflection equation algebras. Using these quantum partial derivatives, we introduce
	the corresponding differential algebras, define the quantum de Rham operator, and exhibit the Leibniz rule for its action. Our
	Leibniz rule differs from its classical analogue and stems from the construction of a quantum double. In the case of the
	algebra $U(gl(N))$, we present another form of the Leibniz rule that is more convenient for prolonging the differential calculus
	to certain extensions of this algebra.
	
	We then use analogues of the Cayley--Hamilton identity for the generating matrices of $U(gl(N))$ and of reflection equation algebras
	to construct projective modules over these algebras. For these projective modules, we introduce Grassmannian connections and
	define the corresponding Chern classes.
\end{abstract}

\maketitle

\section{Introduction}\label{sec:introduction}

In this  paper we develop a new way of constructing a differential calculus on some noncommutative (NC) algebras, which
differs substantially from the known approaches but remains close to the usual calculus on the commutative algebras
$\mathrm{Sym}(gl(N))$. Moreover, our differential algebra reduces to the classical one whenever an NC algebra under
consideration is a deformation of its classical commutative counterpart and tends to it. Besides, we compute the
Chern classes of some projective modules over the mentioned NC algebras in the spirit of \cite{Gom}.

There exist various approaches to an NC differential calculus. One  of them is motivated by the
Hochschild-Kostant-Rosenberg theorem (see \cite{L}). It treats differential forms on an NC algebra in terms of the homology of the
corresponding Hochschild complex, whereas  the role of the de Rham operator is played by the Connes operator $B$. There is
also a cyclic version of this scheme.

Another approach deals with calculus on algebras, related to  quantum $R$-matrices. We  mention only a few representative
papers  \cite{BM, FP, P, WZ, W}, in which the role of the basic NC algebra is played by the so-called RTT algebra (or by its subalgebra)
while the role of the one-sided vector fields is attributed to the corresponding Reflection Equation (RE) algebra. In \cite{GPS2} the RTT
algebra was substituted by another copy of the RE algebra  (denoted as $\MM(R)$ below).

On the algebra $\MM(R)$  and on its modified version one can define analogues of partial derivatives with respect
to the algebra generators, and by passing to an appropriate limit, one gets analogues of the partial derivatives on the enveloping algebra
$U(gl(N))$. We call all these analogues the Quantum Partial Derivatives (QPD). They are deformations of the usual partial derivatives with
respect to the generators of the commutative algebras $\mathrm{Sym}(gl(N))$. Moreover, the quantum  algebra, generated by the RE
algebra $\MM(R)$ or by its limit $U(gl(N))$ together with the corresponding QPD, is a deformation of the usual Heisenberg-Weyl
algebra of $gl(N)$ type (see \cite{GPS2}).

With the use of the QPD we succeeded in constructing the corresponding differential algebra and an analog of the de Rham operator
$\dd$. It should be emphasized that our operator $\dd$ does \emph{not} satisfy the classical Leibniz rule
\begin{equation}
	\dd(a\, b)=\dd(a)\, b+a\, \dd (b).
	\label{Le}
\end{equation}
The same feature is characteristic for the QPD  as well.\footnote{All papers devoted to constructing NC differential calculus can be divided in two
	families: those preserving the Leibniz rule \eqref{Le} (e.g. \cite{BM})   and those neglecting it (e.g. \cite{FP}). Observe that an attempt to
	preserve the classical Leibniz rule while deforming a  commutative algebra $A$ does not allow to get a differential algebra, which could
	be a deformation of the differential algebra $\Om(A)$, except trivial examples. }

It should be stressed that the main difference between our approach and that which is traditionally used in papers devoted to  differential calculus  on an NC
background (see \cite{DVL} or \cite{La}) consists in the following. We consider the space of differential forms as a one-sided module over the basic NC algebra, whereas
the classical Leibniz rule requires a two-sided module structure.
We convert this module into a two-sided module only at the last step in order to present the final result (i.e. the Chern classes) in a special form  which is called canonical.

The main objects of our approach are the algebras $U(gl(N))$ and their $q$-deformations. By these we mean the modified RE algebras,
corresponding to skew-invertible Hecke-type braidings (called Hecke symmetries). Note that our considerations are valid for any Hecke
symmetry $R$, but if, in addition, $R$ is a deformation of the usual flip, then the corresponding modified RE algebra is a two-parameter
deformation of the algebra $\mathrm{Sym}(gl(N))$.

All mentioned algebras have two very specific properties. First, they permit introducing analogues of partial derivatives with respect to
the generators. Let us observe  that the
possibility to introduce QPD is due to the fact that the role of  ``functional'' algebras is attributed
to the RE algebra $\MM(R)$ or to the enveloping algebra $U(gl(N))$ as a specific limit of $\MM(R)$.

Second, these algebras allow constructing a number of projective modules via a matrix Cayley-Hamilton identity\footnote{This identity
	is called \emph{basic}. Besides, there are \emph{higher} Cayley-Hamilton identities constructed in \cite{GS0}. They also give rise
	to projective modules.}
\begin{equation}
	M^N-a_1\, M^{N-1}+...+(-1)^{N-1} a_{N-1}\, M+(-1)^{N} a_N I=0.
	\label{CH0}
\end{equation}
Here $M$ is a matrix composed of generators of $\MM(R)$ or $U(gl(N))$, while the coefficients $a_i$ are central elements in
the corresponding algebra. The symbol $I$ always stands for the identity matrix of an appropriate size.

By introducing the roots $\{\mu_i\}_{1\leq i\leq N}$ of the corresponding Cayley-Hamilton polynomial and assuming them to be pairwise
distinct, we construct a family of idempotents
\begin{equation}
	P_i=\frac{\prod_{j\not=i} (M-\mu_j\, I)}{\prod_{j\not=i}(\mu_i-\mu_j)}, \qquad P_iP_j = \delta_{ij}P_i.
	\label{Idem}
\end{equation}

Observe that for  the commutative algebra $\mathrm{Sym}(gl(N))$ it is also possible to introduce a  generating matrix $M$
subject to \eqref{CH0}. So, in this case  the eigenvalues $\mu_i$ and  idempotents $P_i$ can be defined in the same way. If we assume
the eigenvalues to be fixed $\mu_i=c_i\in \C$, we get a polynomial algebra restricted on a semisimple orbit
$\mathcal{O}_\mu\subset gl(N)^*$ with respect to the coadjoint action of the group $GL(N)$. In this case the idempotents $P_i$ define projective modules over this algebra. If the eigenvalues $\mu_i$
are not fixed, we get the projective modules  over the skew-ring of  the extended algebra $\mathrm{Sym}(gl(N))[\mu_i]$.

A similar situation is valid in the NC algebras under consideration. However, in the noncommutative case it is not evident what constitutes a
proper extension of the QPD onto the extended algebras since the usual Leibniz rule is not valid. We succeeded in finding a
convenient prolongation of the QPD only for $U(gl(2))$ and for its compact counterpart.

It should be emphasized that according to the famous Serre-Swan approach (see \cite{Se, Sw}), the fibre bundles over
regular algebraic varieties can be described in terms of projective modules over the corresponding algebra of functions. By using the related
idempotents, it is possible to define the Grassmannian connection on these modules. Observe that we use the term {\em Grassmannian connection} in a broad sense without assuming the idempotents under consideration  to satisfy the condition $P^\dag=P$, which is often imposed in this construction (see \cite{La}).

Finally, we apply to these idempotents the definition of the Chern classes going back to \cite{Gom}. Thus, we get the Chern classes on our projective modules, which are very close to the classical classes and coincide  with them in the limit. We need to stress that our projective modules are different from  line bundles over the quantum homogeneous space of invariant elements for the quantum subgroup $SO(2)$ of $SO_q(3)$ of \cite{LaPa} and the corresponding idempotents are also different.

We realize this approach on the NC algebras in the following sequence of steps. In Section~\ref{sec:2} we deal with the compact
form $U(u(2))$ of the algebra $U(gl(2), \C)$. We define QPD on these algebras by using the so-called quantum double and
the coalgebraic structure of the subalgebra generated by the QPD. The latter description yields a form of the Leibniz rule for
QPD that is useful for prolonging the QPD to an extension of $U(u(2))$; this prolongation is constructed in
Section~\ref{sec:spectral}. In Section~\ref{sec:differential} we define a differential algebra $\Om(gl(N))$ stemming from
$U(gl(N))$ and introduce the de Rham operator without using the classical Leibniz rule. In Section~\ref{sec:characteristic} we define
an NC version of the characteristic classes on the extended algebra $U(u(2))$. In Section~\ref{sec:generalizations} we discuss
possible generalizations, define the differential algebra and de Rham operator corresponding to $U(gl(N))$ and the RE algebra,
and conclude with the example of a two-parameter quantum hyperboloid.
The ground field is assumed to be $\C$.

\section[Quantum partial derivatives on U(gl(N)h)]{Quantum partial derivatives on $U(gl(N)_h)$}
\label{sec:2}

Let us consider the Lie algebra $gl(N)_h$ where the parameter $h$ is a multiplier in the right hand side of the standard
$gl(N)$ Lie bracket
\begin{equation}
	[\hl_i^j\,, \hl_k^l]=h\,(\de_k^j\, \hl_i^l-\de_i^l\, \hl_k^j).
	\label{gl-h}
\end{equation}
Here, the family $\{\hl_i^j\}_{1\leq i,j\leq N}$ is the standard   basis composed of  the matrix units, provided $h=1$. Due to the parameter $h$ the enveloping algebra $U(gl(N)_h)$
acquires the meaning of a deformation of the commutative algebra $\mathrm{Sym}(gl(N))$.

To define the QPD, we first introduce a commutative associative algebra $\DD$ finitely generated by $N^2$ elements
$\{\partial_i^j\}_{1\le i,j\le N}$:
\[
\partial_i^j\,\partial_k^n = \partial_k^n\,\partial_i^j,\qquad 1\le i,j,k,n\le N.
\]
It is not difficult to verify that the algebra $\DD$ can be given a \emph{bialgebra} structure with the coproduct $\Delta$
and the counit $\varepsilon$ defined on generators by the following relations:
\begin{align}
	\De(\pa_i^j) &= \pa_i^j\ot 1_\DD+1_\DD\ot \pa_i^j
	+h\sum_{k=1}^N \pa^j_k\ot \pa_i^k, \label{co}\\
	\varepsilon(1_\DD) &= 1, \qquad \varepsilon(\partial_i^j)=0. \label{coun}
\end{align}
Here, $1_\DD$ is the unit element of the algebra $\DD$. As usual, the actions of the coproduct and counit on arbitrary elements of $\DD$
is extended in the natural way, thus obtaining a bialgebra structure. Note that due to the last term in \eqref{co} the coproduct $\Delta$ is \emph{not}
cocommutative.

Now we define a linear action of the algebra $\DD$ on $U(gl(N)_h)$ which is a representation of $\DD$ and, on the other hand,
is compatible with the algebraic structure of the enveloping algebra\footnote{This means that the action of $\DD$ sends to zero the ideal generated by relations \eqref{gl-h}.}.

We set by definition
\begin{equation}
	\partial_i^j(1_{U} ) = 0,\qquad \partial_i^j(\hl_k^l) = \delta_i^l\,\delta_k^j\,1_U,
	\label{gen-act}
\end{equation}
where $1_U$ is the unit element of the universal enveloping algebra $U(gl(N)_h)$. Thus, the action of the elements $\pa_i^j$ on the unit and the generators $m_k^j$ is the same as in the algebra $sym(gl(N))$.

Their actions on higher-degree polynomials in the generators $m_k^j$ are defined via the coproduct \eqref{co}. More precisely, we set
\begin{equation}
	\pa_i^j(a\,b):= (\pa_i^j)_{(1)}(a)\, (\pa_i^j)_{(2)}(b) = \partial_i^j(a)\,b+ a\,\partial_i^j(b) +h\sum_k\partial_k^j(a)
	\partial_i^k(b),
	\label{Lei}
\end{equation}
where $\De(\pa_i^j) = (\pa_i^j)_{(1)}\ot (\pa_i^j)_{(2)}$ is Sweedler's notation for the coproduct \eqref{co}. The unit element
$1_\DD$ acts as the identity operator. The proof that the defined action of the elements $\pa_i^j$ is compatible with the structure of the algebra $U(gl(N)_h)$ is straightforward (see \cite{GPS2} for details).

The relation \eqref{Lei} is a form of the {\em quantum} Leibniz rule. However, it can be expressed in other forms. Thus, if  in the relation \eqref{Lei}
we omit the element $b$, we get the so-called {\em permutation map}
\[\DD\otimes U(gl(N)_h)\rightarrow U(gl(N)_h)\otimes \DD:\]
\begin{align}
	\pa_i^j\ot a &\longmapsto (\pa_i^j)_{(1)}(a)\ot (\pa_i^j)_{(2)} \notag\\
	&= \partial_i^j(a)\otimes 1_\DD+a\otimes \partial_i^j
	+h\sum_k\partial_k^j(a)\otimes \partial_i^k, \label{first}\\
	1_\DD\otimes a &\longmapsto a\otimes 1_\DD, \qquad
	\omega\otimes 1_U\longmapsto 1_U\otimes \omega, \notag
\end{align}
where $a\in U(gl(N)_h)$ and $\omega\in\DD$.

Moreover, the permutation map defines a \emph{quantum double} $(\DD,U(gl(N)_h))$. This is the associative unital algebra generated by
$2N^2$ elements $m_i^j\in U(gl(N)_h)$ and $\partial_k^l\in\DD$, subject to the \emph{permutation relations} that follow
from \eqref{first} and \eqref{gen-act}:
\begin{equation}
	\pa_i^{\,j}\, m_k^l = m_k^l\,\partial_i^{\,j} +\delta_i^l\delta_k^{\,j}+ h\,\delta_k^{\,j}\,\partial_i^l.
	\label{perrel}
\end{equation}
If we introduce $N\times N$ matrices $M = \|m_i^{\,j}\|$ and $D = \|\partial_i^{\,j}\|$, then the above permutation relations among the
generators of the double can be written in a matrix form:
\[
D_1M_2 - M_2D_1 = P_{12} +hP_{12}D_2,
\]
where $P_{12}$ is the matrix of the flip and $D_1 = D\otimes I$, $M_2 = I\otimes M$, $I$ being $N\times N$ unit matrix. Note that
the algebras $U(gl(N)_h)$ and $\DD$ are the proper subalgebras of the quantum double $(\DD, U(gl(N)_h))$.

\begin{remark}
	\label{rem:1}
	A more general construction of a quantum double, related to the RE algebras will be considered in section 6. Here we only note that,
	given a quantum double $(A,B)$ of two associative algebras $A$ and $B$, we can equip the vector space $B\otimes A$ with
	the structure of an associative algebra provided we have a permutation map $\sigma: \,A\otimes B\rightarrow B\otimes A$ which preserves
	the ideals generated by the defining relations of both algebras. An evident example of the permutation map existing for any $A$ and
	$B$ is the usual flip: $a\otimes b\mapsto b\otimes a$. Such a permutation map defines a trivial structure of the quantum double.
	
	Given a quantum double $(A, B)$, in order to convert the elements of the algebra $A$ into operators acting on the algebra $B$ we require a counit $ \varepsilon:\,A\to \C$.
	Then the action of an element $a\in A$ on $b\in B$ is defined as follows. First, we permute
	the factors in $a\ot b$  with the use of the permutation map and then apply the counit to the right factors, belonging to $A$. The result of
	this procedure $a(b)$ belongs to $B$ (we identify $B$ and $B\ot \C$ in the natural way):
	\[
	a(b) := (\mathrm{id}_B\otimes \varepsilon)\circ\sigma(a\otimes b).
	\]
	In the case of the trivial permutation map (the usual flip) the action reduces to multiplication by a complex number:
	$a(b) = \varepsilon(a) b$.
\end{remark}

Now, consider the case $N=2$ in more detail. It is convenient to pass to the compact form $u(2)_h$ of the Lie algebra $gl(2)_h$.
For this purpose we define new generators $x,y,z$ and $t$:
\[
m_1^1=t-iz,\quad m_1^2=-ix-y,\quad m_2^1=-ix+y,\quad m_2^2=t+iz.
\]
Then the Lie brackets of the new generators read:
\[
[x, \, y]= h\, z,\quad [y, \, z]= h\, x,\quad[z, \, x]=h\, y,\quad [t, \, x]=[t, \, y]=[t, \, z]=0.
\]

Thus, the element $t$ is central in the enveloping algebra $U(u(2)_h)$. Another central element is the quadratic polynomial
\[
{{\Cas}} = x^2+y^2+z^2.
\]
Moreover, $t$ and $\Cas$ generate the whole center $Z(U(u(2)_h))$ of the algebra $U(u(2)_h)$.

The permutation relations among the partial derivatives $\partial_t$, $\partial_x$, $\partial_y$, $\partial_z$ and the
generators $\{t,x,y, z\}$ were derived in \cite{GPS2}. To simplify the formulae we introduce the {\em shifted derivative} in $t$
$\tilde \partial_t = \partial_t +\frac{2}{h}\,1_{\DD}$:
\begin{equation}
	\begin{array}{l@{\quad}l@{\quad}l@{\quad}l}
		\tilde\pa_t\,t - t\,\tilde\pa_t = \frac{h}{2}\,\tilde\pa_t & \tilde\pa_t\, x - x\,\tilde\pa_t
		=-\frac{h}{2}\,\pa_x &
		\tilde\pa_t\, y - y\, \tilde\pa_t=-\frac{h}{2}\,\pa_y &\tilde\pa_t\, z - z\,\tilde\pa_t=- \frac{h}{2}\,\pa_z\\
		\rule{0pt}{7mm}
		\pa_x\, t - t\,\pa_x = \frac{h}{2}\,\pa_x &\pa_x \,x -  x\,\pa_x = \frac{h}{2}\,\tilde\pa_t &
		\pa_x \, y-  y\,\pa_x = \frac{h}{2}\,\pa_z & \pa_x \,z - z\, \pa_x  = - \frac{h}{2}\,\pa_y \\
		\rule{0pt}{7mm}
		\pa_y \,t - t \, \pa_y = \frac{h}{2}\,\pa_y & \pa_y \,x -  x\,  \pa_y = -\frac{h}{2}\,\pa_z &
		\pa_y \,y - y \,  \pa_y = \frac{h}{2}\,\tilde\pa_t & \pa_y \,z - z \,  \pa_ y= \frac{h}{2}\,\pa_x\\
		\rule{0pt}{7mm}
		\pa_z \,t - t \,\pa_z = \frac{h}{2}\,\pa_z & \pa_z \,x - x \,\pa_z = \frac{h}{2}\,\pa_y&
		\pa_z \,y -  y\,\pa_z = -\frac{h}{2}\,\pa_x & \pa_z \,z - z \,\pa_z = \frac{h}{2}\,\tilde\pa_t.
	\end{array}
	\label{perm}
\end{equation}

Extending the above permutation relations to the whole algebras $\DD$ and $U(u(2)_h)$, we get a permutation map
\[\sigma: {\DD}\ot U(u(2)_h)\to U(u(2)_h)\ot {\DD}.\]
Below,  we omit the symbol $\ot$ and the units $1_{U}$ and $1_{\DD}$  when this does not lead to a misunderstanding.

Also, consider the counit $\varepsilon:\DD\to \C$ defined in accordance with \eqref{coun}:
\begin{equation}
	\varepsilon(1_{\DD})=1, \quad \varepsilon(\pa_u)=0,\quad \forall\,\, u\in \{t,x,y,z\}.
	\label{counn}
\end{equation}
Consequently, for the shifted derivative $\tilde\partial_t$ we get $\varepsilon(\tpa_t)=\frac{2}{h}$.
Then, in order to find the action of a partial derivative $\pa_u$ on an element $b\in  U(u(2)_h)$ we permute the elements
$\pa_u$ and $b$ in the product
$\pa_u\ot b$ with the use of the relations \eqref{perm} and apply the above counit map $\varepsilon$ to the right factors,
belonging to $\DD$. Thus, we get the resulting
element $\pa_u(b)$.

This definition immediately entails that, for example,
\[
\pa_x(x)=1, \quad \pa_x(y)= \pa_x(z)=\pa_x(t) = 0
\]
and so on, that is the partial derivatives act on the generators of the algebra $U(u(2)_h)$ in the classical way.
Also, the result of action of the derivative $\pa_x$ on the quadratic monomial $xy$ is still classical:
\[ \pa_x(xy)= y.\]

However, in general it is not so, as the following examples show. Using the permutation relations, we get:
\[
\pa_x (yz)=(y\pa_x+\frac{h}{2}\pa_z)z=y(z\pa_x-\frac{h}{2}\pa_y)+\frac{h}{2}(z\pa_z+\frac{h}{2}\tpa_t)=
yz\pa_x-\frac{h}{2}y\pa_y+\frac{h}{2}z\pa_z+\frac{h^2}{4}\tpa_t.
\]
Now, on applying the counit to factors from $\DD$, we get $\pa_x (yz)=\frac{h}{2}$.
In the same way we get $\pa_x (zy)=-\frac{h}{2}$. These results are compatible with the relation  $yz-zy=hx$.

The coproduct $\De$ applied to the basis elements $\{\pa_t,\,\pa_x,\,\pa_y,\,\pa_z\}$ reads:
\begin{equation}
	\begin{array}{l}
		\De(\pa_t)=\pa_t\ot 1+ 1\ot \pa_t+\frac{h}{2}(\pa_t\ot \pa_t-\pa_x\ot \pa_x-\pa_y\ot \pa_y-\pa_z\ot \pa_z),\\
		\rule{0pt}{5mm}
		\rule{0pt}{5mm}
		\De(\pa_x)=\pa_x\ot 1+ 1\ot \pa_x +   \frac{h}{2}(\pa_t\ot \pa_x+\pa_x\ot \pa_t+\pa_y\ot \pa_z-\pa_z\ot \pa_y),\\
		\rule{0pt}{5mm}
		\De(\pa_y)=\pa_y\ot 1+ 1\ot \pa_y+  \frac{h}{2}(\pa_t\ot \pa_y+\pa_y\ot \pa_t+\pa_z\ot \pa_x-\pa_x\ot \pa_z),\\
		\rule{0pt}{5mm}
		\De(\pa_z)=\pa_z\ot 1+ 1\ot \pa_z+ \frac{h}{2}(\pa_t\ot \pa_z+\pa_z\ot \pa_t+\pa_x\ot \pa_y-\pa_y\ot \pa_x).
	\end{array}
	\label{arr0}
\end{equation}
The use of the shifted derivative $\tpa_t$ simplifies the above relations to the form:
\begin{equation}
	\begin{array}{l}
		\De(\tpa_t)=\frac{h}{2}(\tpa_t\ot \tpa_t-\pa_x\ot \pa_x-\pa_y\ot \pa_y-\pa_z\ot \pa_z),\\
		\rule{0pt}{5mm}
		\rule{0pt}{5mm}
		\De(\pa_x)=\frac{h}{2}(\tpa_t\ot \pa_x+\pa_x\ot \tpa_t+\pa_y\ot \pa_z-\pa_z\ot \pa_y),\\
		\rule{0pt}{5mm}
		\De(\pa_y)=\frac{h}{2}(\tpa_t\ot \pa_y+\pa_y\ot \tpa_t+\pa_z\ot \pa_x-\pa_x\ot \pa_z),\\
		\rule{0pt}{5mm}
		\De(\pa_z)=\frac{h}{2}(\tpa_t\ot \pa_z+\pa_z\ot \tpa_t+\pa_x\ot \pa_y-\pa_y\ot \pa_x).
	\end{array} \label{arr2}
\end{equation}
These formulae motivate us to introduce   the following matrix
\begin{equation} {\Theta}=\left(\begin{array}{llll}
		\tpa_t&-\pa_x&-\pa_y&-\pa_z\\
		\pa_x&\,\,\,\,\,\tpa_t&-\pa_z&\,\,\,\,\,\pa_y\\
		\pa_y&\,\,\,\,\,\pa_z&\,\,\,\,\,\tpa_t&-\pa_x\\
		\pa_z&-\pa_y&\,\,\,\,\,\pa_x&\,\,\,\,\,\tpa_t \end{array} \right). \label{seven}
\end{equation}
It allows us to present \eqref{arr2} in the following matrix form:
\[
\De ({\Theta})=\frac{h}{2}\, {\Theta} \stackrel{.}{\otimes} {\Theta},
\]
where the notation $A\stackrel{.}{\otimes} B$ stands for the matrix with entries $(A\stackrel{.}{\otimes} B)_i^j=\sum_k\, A_i^k\ot B_k^j $.

This entails the following important consequence:
\[
{\Theta}(ab)=\frac{h}{2} \,{\Theta}(a){\Theta}(b)=i\hh\, {\Theta}(a){\Theta}(b),\quad \forall\, a,b \in U(u(2)_h),
\]
where we use  the renormalized parameter $h = 2i\,\hh$. The notation ${\Theta}(a)$ means that each entry of the
matrix $\Theta$ is applied to $a$.

Practically, it is more useful to deal with the matrix  $\hatt=i\hh \Theta$. It also defines a map
\[
U(u(2)_h) \ni a \mapsto \hatt(a)\in {\Mat}(U(u(2)_h),
\]
which is a homomorphism of the algebra $U(u(2)_h)$ into the algebra ${\Mat}(U(u(2)_h))$ of $4\times 4$ matrices
with entries belonging to $U(u(2)_h)$. Namely, we have
\begin{equation}
	{\hatt}(ab)={\hatt}(a){\hatt}(b),\quad \forall \,a,b \in U(u(2)_h)
	\label{The}
\end{equation}
and $\hatt(1_U)=I$. Hereafter, the notation $\hatt$ stands  for the matrix and for the corresponding map
from $U(u(2)_h)$ to ${\Mat}(U(u(2)_h))$.

As an example we present the images of the elements $x,y,z$ and $\Cas$ under the map $\hatt$:
\begin{align}
	{\hatt}(x)&=
	\left(\!\!\begin{array}{@{}cccc@{}}
		x&-i\hh&0&0\\
		i\hh&x&0&0\\
		0&0&x&-i\hh\\
		0&0&i\hh&x
	\end{array}\!\!\right), \notag\\[1ex]
	{\hatt}(y)&=
	\left(\!\!\begin{array}{@{}cccc@{}}
		y&0&-i\hh&0\\
		0&y&0&i\hh\\
		i\hh&0&y&0\\
		0&-i\hh&0&y
	\end{array}\!\!\right), \notag\\[1ex]
	{\hatt}(z)&=
	\left(\!\!\begin{array}{@{}cccc@{}}
		z&0&0&-i\hh\\
		0&z&-i\hh&0\\
		0&i\hh&z&0\\
		i\hh&0&0&z
	\end{array}\!\!\right).
\end{align}
\begin{equation}
	{\hatt}(\Cas)=
	\left(\!\!
	\begin{array}{@{}cccc@{}}
		\Cas+3\hh^2 &   \displaystyle -2i\hh x\,& \displaystyle -2i\hh y &\displaystyle -2i\hh z\\
		\displaystyle  2i\hh x&\Cas+3\hh^2&\displaystyle -2i\hh z &\displaystyle 2i\hh y \\
		\displaystyle 2i\hh y &\displaystyle 2i\hh z &\Cas+3\hh^2 &\displaystyle -2i\hh x\\
		\displaystyle 2i\hh z &\displaystyle -2i\hh y & \displaystyle 2i\hh x&\Cas+3\hh^2
	\end{array}
	\!\!\right).
\end{equation}
Also, note that  $\hatt(t^p)=(t+i\hh)^p I,\,\, \forall p\in \mathbb{Z }$.

Thus, we have presented a few forms of the Leibniz rule for the QPD. One of them is realised in \eqref{perrel} via the permutation relations of
the QPD and elements of the algebra $U(u(2)_h)$. It is the most general and will be used below in the differential calculus on the RE algebras.
The other one is based on the bi-algebra structure of the commutative algebra generated by the QPD. And the last one (in fact it is another presentation of the previous one)
is encoded in the multiplicative property of the map $\hatt$. This
property is very useful for a prolongation of the QPD onto some extensions of the algebra $U(u(2)_h)$, described in the next section.

\section[QPD on a spectral extension of U(u(2)h)]{QPD on a spectral extension of the algebra $U(u(2)_h)$}\label{sec:spectral}

Now, consider the matrix
\begin{equation}
	M=\left(\begin{array}{cc}
		m_1^1& m_1^2\\
		m_2^1& m_2^2
	\end{array}\right)=\left(\begin{array}{cc}
		t-i\,z& -i\, x-y\\
		-i\, x+y& t+i\, z
	\end{array}\right).
	\label{matrN}
\end{equation}
with entries belonging to the algebra $U(gl(2)_h)$. It is called the \emph{generating matrix} of the algebra $U(gl(2)_h)$.
The permutation relations among the generators can be presented in the form of matrix equality:
\[
P\,M_1P\,M_1-M_1P\,M_1P=h\,(P\, M_1-M_1 P), \qquad M_1=M\ot I,
\]
where $P$ is the matrix of the usual flip.

Note that $M$ satisfies the following analogue of the matrix Cayley-Hamilton identity:
\begin{equation}
	\chi(M) := M^2-(2t+h)\, M+\,(t^2+x^2+y^2+z^2+h\, t)\,I= 0.
	\label{CH}
\end{equation}

It is natural to introduce the roots $\mu_1,\, \mu_2$ of the polynomial $\chi(s)$:
\[
\mu_1+\mu_2=(2t+h),\qquad \mu_1 \mu_2=t^2+x^2+y^2+z^2+h\, t.
\]
Since the coefficients of the polynomial $\chi(M)$ are central elements of $U(u(2)_h)$, we assume the elements $\mu_1$ and
$\mu_2$ to be central in the extended algebra $U(u(2)_h)[\mu_1,\mu_2]$.

It is convenient to introduce the difference $\mu=\mu_1-\mu_2$.
The following relations hold true:
\begin{equation}
	t=\frac{\mu_1+\mu_2- h}{2},\qquad {{\Cas}}=\frac{h^2-\mu^2}{4}.
	\label{cas}
\end{equation}
Also, we need the quantity
\begin{equation}
	\rh=\frac{\mu}{2i}=\sqrt{\Cas+\hh^2}  = \sqrt{x^2+y^2+z^2+\hh^2},
	\label{sign}
\end{equation}
which is called {\em the quantum radius}. Recall that $h=2i\hh$. 

Let $\C(t,\rh) $ stand for the field of all rational functions in the elements $t$ and $\rh$, which are central in the algebra
$U(u(2)_h)[\mu_1,\mu_2]$. Consider the quotient algebra
\begin{equation}
	\AA=\left(U(su(2)_h)\ot \C(t,\rh)\right)/\langle x^2+y^2+z^2-\rh^2+\hh^2 \rangle,
	\label{quot}
\end{equation}
where $\langle J \,\rangle$ stands for the two-sided ideal generated by a subset $J$, and $su(2)_h\subset u(2)_h$ is a Lie
subalgebra generated by $x$, $y$ and $z$.

The quantum radius is central in the  algebra $\AA$, since it is so for the quantity $\mu$.
In order to fix the sign of the square root in \eqref{sign}, we assume  the quantum
radius to be positive provided $\hh$ is real and $x,y,z$ are  represented by Hermitian operators.

In \cite{GS2} we extended the QPD to the algebra $\AA$. This enables us to compute the image $\hatt(a)$
of any element  $a\in \AA$. For instance, the matrix $\hatt(\rh)$ reads:
\begin{equation} {\hatt}(\rh)=
	\frac{1}{\rh}\left(\!\!
	\begin{array}{cccc}
		\rh^2+\hh^2 &   \displaystyle -i\hh x\,& \displaystyle -i\hh y &\displaystyle -i\hh z\\
		\displaystyle  i\hh x&\rh^2+\hh^2&\displaystyle -i\hh z &\displaystyle i\hh y \\
		\displaystyle i\hh y &\displaystyle i\hh z & \rh^2+\hh^2 &\displaystyle -i\hh x\\
		\displaystyle i\hh z &\displaystyle -i\hh y & \displaystyle i\hh x& \rh^2+\hh^2
	\end{array}
	\!\!\right).
	\label{mat1}
\end{equation}

It is worth noticing that  the action of the map $\hatt$ onto the commutative algebra $\C(t,\rh)$ is subject to the natural rule:
\[
\hatt(f(t,\rh))=f(\hatt(t), \hatt(\rh)),\qquad \forall\, f\in \C(t,\rh).
\]

\begin{proposition}
	The map $\hatt: \AA\to {\Mat}(\AA)$ is well defined, that is the matrices  $\hatt(t)$, $\hatt(\rh)$ commute with $\hatt(a)$
	for any $ a\in U(sl(2)_h)$. Moreover, the map $\hatt$ is compatible with the relation ${\Cas}=\rh^2-\hh^2$:
	\begin{equation}
		\hatt({\Cas})=\hatt({\rh})^2-\hh^2 {I}.
		\label{ree}
	\end{equation}
\end{proposition}

\begin{proof}
	Both claims follow by straightforward computations.
\end{proof}

Now, we turn to the second step of extending the algebra $U(gl(2)_h)$ and consider the skew-field $\Ah=\AA[\AA^{-1}]$. This
skew-field consists of left fractions $a^{-1} b$, $a,b \in\AA$, $a\not=0$, which by the Ore property can be presented as
right fractions $c\,d^{-1}$. We want to present a way of extending the QPD onto the algebra $\Ah$.

First, recall the method of solving the same problem in a commutative unital algebra $A$.
Let  $V:A\to A$ be a vector field,  i.e. an operator  subject to the usual Leibniz rule.
Then applying $V$ to the element $a\,a^{-1}=1$ (provided that $a$ is invertible) we get
\[
V(a)\,a^{-1}+a\,V(a^{-1})=0\quad \Rightarrow\quad V(a^{-1})=- a^{-1}\, V(a)\, a^{-1}.
\]

However, the classical Leibniz rule is not valid in the algebra $U(u(2)_h)$ any more. Instead, we have a new form of the Leibniz rule,
which consists in the multiplicative property of the map $\hatt$. In order to preserve this property for the algebra $\Ah$, we set
\emph{by definition} $\hatt(a^{-1})=\hatt(a)^{-1}$, provided the matrix $\hatt(a)$ is invertible in the algebra ${\Mat}(\Ah)$.

Observe that the matrix $\hatt(\rh)$ in \eqref{mat1} is invertible. Moreover, it is not difficult to find the matrix $\hatt(\rh^p)=i\hh \Theta(\rh^p)$,
$p\in \mathbb{Z}$. Its entries can be easily computed according to the following formulae (see  \cite{GS2}, formula (2.12))
\begin{equation}
	\tpa_t(\rh^p)=\frac{-i}{2\hh\rh}((\rh+\hh)^{p+1}+(\rh-\hh)^{p+1}).
	\label{pp}
\end{equation}
Consequently, we have
\begin{equation}
	\pa_t(\rh^p)=\frac{i}{2\hh\rh}(2\rh^{p+1}-(\rh+\hh)^{p+1}-(\rh-\hh)^{p+1}).
	\label{pppp}
\end{equation}
It is not difficult to see that for $\hh\to 0$ this derivative vanishes.

As was shown in \cite{GS2}, the following formulae are valid:
\begin{equation}
	\pa_x(\rh^p)=\frac{x}{\rh}\,\parh(\rh^p), \qquad \pa_y(\rh^p)=\frac{y}{\rh}\,\parh(\rh^p),\qquad  \pa_z(\rh^p)=\frac{z}{\rh}\,\parh(\rh^p).
	\label{ppp}
\end{equation}
Note that above relations are identical to the classical ones. In particular, for $p=-1$ we have
\[
\parh(\rh^{-1})=\frac{1}{\hh^2-\rh^2},\qquad \pa_x(\rh^{-1})=\frac{x}{\rh(\hh^2-\rh^2)},\qquad
\tpa_t(\rh^{-1})=\frac{-i}{\hh \rh}\quad\Rightarrow\quad \partial_t(\rh^{-1}) = 0.
\]

Also, in \cite{GS2} the derivative with respect to the quantum radius was introduced by the rule:
\[
\parh(f(\rh))=\frac{f(\rh+\hh)-f(\rh-\hh)}{2\hh}.
\]
\begin{remark}
	Observe that the ordering of the eigenvalues $\mu_i$ is arbitrary. So, the quantity $\mu$ and consequently $\rh$ is defined up
	to a sign. However, the	quantity $\frac{\pa_{\rh}}{\rh}$ is invariant under the change $\rh\mapsto -\rh$.
\end{remark}

Now, we are able to calculate the matrix $\hatt(\rh^p)$, $p \in \mathbb{Z}$. We have
\begin{equation}
	\hatt(\rh^p)=\frac{(\rh+\hh)^{p+1}+(\rh-\hh)^{p+1}}{2\rh}\, I-\frac{i((\rh+\hh)^{p}-(\rh-\hh)^{p})}{2\rh}\, X,
	\label{ten}
\end{equation}
with the following matrix:
\[
X=
\left(\!\!
\begin{array}{cccc}
	0 & -x\,&-y&-z\\
	x& 0 &-z&y\\
	y& z & 0 & -x\\
	z&-y&  x& 0
\end{array}
\!\!\right).
\]

It can be easily checked that
\[
\hatt(\rh^p)\,\hatt(\rh^q)=\hatt(\rh^{p+q}),\qquad \forall\, p,q \in \mathbb{Z}.
\]

\section{Differential forms and de Rham operator via QPD}\label{sec:differential}

Our current aim is to incorporate  differentials in our calculus. We show that once the QPD are defined, the differentials can be also
easily defined. In this section we are dealing with the general case $U(gl(N)_h)$.

Let us introduce $N^2$ new elements $d_i^j$, $1\le i,j\le N$, which are assumed to be skew-com\-mu\-ta\-tive:
\[
d_i^j\, d_k^l=-d_k^l\, d_i^j,\qquad 1\leq i,j,k,l \leq N.
\]
They generate a skew-symmetric algebra which will be denoted $\Om=\Om(gl(N))$. Elements of this algebra will be called
{\em pure differentials}. This algebra is a direct sum of its $k$-th order homogeneous components $\Om^k(gl(N))$, $0\le k\le N^2$.

The tensor product
\begin{equation}
	\Om(gl(N))\ot U(gl(N)_h) \label{produc}
\end{equation}
is a right $U(gl(N)_h)$-module.

Define the linear de Rham operator $\dd: \Omega\otimes U(gl(N)_h)\rightarrow \Omega\otimes U(gl(N)_h)$ as
follows:
\begin{equation}
	\dd(\om\ot a):= \om\, \sum_{i, j} d_i^j\ot \pa_j^i(a),\qquad \om\in \Om(gl(N)),\quad a\in \UU.
	\label{dR-act}
\end{equation}
In other words, the sum $\sum_{i, j} d_i^j\ot \pa_j^i$ is placed  between the elements $\om\in \Om$ and $a\in U(gl(N)_h)$ and  the
QPD are applied to the  element $a$. On other elements we extend this operator by linearity.

Taking into account \eqref{gen-act} we immediately get $\dd(m_i^j)= d_i^j$. Thus, the element $d_i^j$ is  the differential of the
generator $m_i^j$.

\begin{proposition}
	The following holds: $ \dd^2=0$.
\end{proposition}

\begin{proof}
	If we apply the operator $\dd$ twice, we get:
	\[
	\dd^2(\om\ot a):=  \om\, \sum_{i, j} \sum_{k, l}d_i^j d_k^{\,l} \ot \pa_l^k(\pa_j^i(a)).
	\]
	Now, we use the fact that the elements $d_i^j$ and $d_k^{\,l}$ are skew-commutative, so their product in the algebra $\Om$ can
	be identically presented in the skew-symmetrized form:
	\[
	d_i^jd_k^{\,l} = \frac{1}{2}\,(d_i^jd_k^{\,l}-d_k^{\,l}d_i^j).
	\]
	In a similar manner we can symmetrize the product of commutative QPD:
	\[
	\pa_l^k\pa_j^i = \frac{1}{2}\,(\pa_l^k\pa_j^i+\pa_j^i \pa_l^k).
	\]
	By multiplying these respectively skew-symmetrized and symmetrized expressions we obtain zero identically:
	\[
	\sum_{i,j,k,l}
	d_i^j d_k^{\,l} \ot \pa_l^k\pa_j^i\equiv 0.
	\]
\end{proof}

It should be emphasized that by \eqref{dR-act} the de Rham operator $\dd$ is defined on the elements $\om \ot a$. If we want to apply
$\dd$ to an element $a\ot \om$, we have to transform it to $\om \ot a$ (we call this {\em the canonical form}) and then apply
the de Rham operator in accordance with \eqref{dR-act}. This transformation is possible provided we permute the factors by means
of the usual flip. Otherwise stated, we identify $a\otimes \omega$ with $\omega\otimes a $ for any $\om\in \Om(gl(N))$ and
$a\in  U(gl(N)_h)$. This means that we impose the permutation relations $\om\ot a=a\ot\om$.

Also, by using these permutation relations  we can define the following product
\[
(\om_1\ot a_1)\cdot (\om_2\ot a_2):= \om_1\wedge \om_2\ot (a_1a_2)
\]
on $\Om(gl(N))\ot U(gl(N)_h)$.

\begin{remark}
	The permutation relations between elements of $\Omega$ and $U(gl(N)_h)$ do not affect the Leibniz rule for the
	de Rham operator $\dd$, since it is always applied only to elements written in the canonical form. This is the reason why
	in the definition \eqref{dR-act} of the de Rham operator we do not introduce the factor $(-1)^p$ for $\om\in \Om^p(gl(N))$ which
	is usual for the classical differential calculus.
	In the classical case the relation $a\, \dd b=\dd b\, a$ is assumed and by applying the de Rham operator to this relation we get a contradiction if the factor
	$(-1)^p$ is omitted.
	
	However, in our scheme we have to apply the relation $\om\ot a=a\ot \om$ only on the final step, if we want to reduce the differential form to the canonical presentation.
\end{remark}

Consider now the low dimensional example with the algebra $U(u(2)_h)$. In this case the de Rham operator $\dd$ acts on
the space $\Om(u(2)_h)\ot  U(u(2)_h)$ and explicitly can be written as (see Section \ref{sec:2} for definitions):
\begin{equation}
	\dd=\dd t \, \pa_t+\dd x \,\pa_x+\dd y\, \pa_y+\dd z\, \pa_z.
	\label{clas}
\end{equation}
In the same manner as above this sum should be placed between elements from $\Om$ and $U(u(2)_h)$ and the QPD should be
applied to the right factor.

Let us recall that under a linear change of generators, the corresponding QPD and the differentials are  transformed in the classical way.
So, by passing to the basis $\{t, x, y, z\}$  (see \ref{matrN}), we get \eqref{clas}.

As an example  we explicitly compute $\dd(yz)$:
\[
\dd(yz)=\dd t \, \pa_t(yz)+\dd x \,\pa_x(yz)+\dd y\, \pa_y(yz)+\dd z \pa_z(yz)=
\dd y\, z+\dd z \, y +\frac{h}{2}\, \dd x.
\]
This result differs from its classical counterpart but coincides with it in the limit $h=0$.

Also, we have
\[
\dd(x^2)=2 \, \dd x\, x-\frac{h}{2}\, \dd t.
\]

\section[Characteristic classes on the extended U(u(2)h)]{Characteristic classes on the extended $U(u(2)_h)$}\label{sec:characteristic}

In this section we consider the so-called \emph{Grassmannian connection} which is well-defined on the projective modules. Here
we constrain ourselves with considering projective modules over the enveloping algebra $U(u(2)_h)$.

Let $\mathcal{W}$ be a free $U(u(2)_h)$ module and $\mathcal{V}\subset \mathcal{W}$ be a projective submodule corresponding to
a projector $P$: $\mathcal{V} = P(\mathcal{W})$. Then by definition the Grassmannian connection is the composition
of operations $P\circ \dd$. According to \cite{Gom} the corresponding first Chern class ${\Ch_1(P)}$ is given by the trace:
\[
{\Ch_1(P)} := \Tr\, (P \dd P \dd P).
\]

Note that the generating matrix $M$ (see  \eqref{matrN}) of the algebra $U(u(2)_h)$ satisfies the second-degree Cayley-Hamilton identity:
\[
(M-\mu_1I)(M-\mu_2I) = 0.
\]
The corresponding idempotents are
\[
P_1=\frac{M-\mu_2\, I}{\mu_1-\mu_2},\qquad P_2 =\frac{M-\mu_1\, I}{\mu_2-\mu_1}.
\]
The idempotents are complementary: $P_1P_2=0$, $ P_1+P_2=I$.

In order to compute
\[
\dd P_1=\dd t \,\partial_tP_1 +\dd x \,\partial_xP_1 +\dd y \,\partial_yP_1 + \dd z \,\partial_zP_1
\]
we have to find the partial derivatives of $P_1$ with respect to generators $t$, $x$, $y$ and $z$. For this purpose,
we first express $\mu_2$ in terms of $t$ and $\rh$. On taking into account \eqref{cas} and \eqref{sign}, we get:
\[
\left\{
\begin{array}{l}
	\mu_1+\mu_2=2t+2i\hh\\
	\rule{0pt}{5mm}
	\mu_1-\mu_2=\mu=2i \rh
\end{array}
\right.\quad \Rightarrow\quad
\mu_2=t - i\rh +i\hh.
\]
So, we rewrite $P_1$ in the form:
\[
P_1=\frac{M-\mu_2\,I}{\mu_1-\mu_2}=\frac{M-(t+i\hh-i\rh)\,I}{2i\rh} = \frac{M}{2i\rh}+\left(\frac{1}{2}-\frac{t}{2i\rh}-\frac{\hh}{2\rh}\right) I.
\]
Taking into account the formulae
\[
\pa_t(M)=\left(\!\!\begin{array}{cc}
	1& 0\\
	0& 1
\end{array}\!\!\right),\quad \pa_x(M)=\left(\!\!\begin{array}{cc}
	0& -i\\
	-i& 0
\end{array}\!\!\right),\quad \pa_y(M)=\left(\!\!\begin{array}{cc}
	0& -1\\
	1& 0
\end{array}\!\!\right),\quad \pa_z(M)=\left(\!\!\begin{array}{cc}
	-i& 0\\
	0& i
\end{array}\!\!\right),
\]
the coproduct \eqref{arr0} and relations \eqref{pppp} and \eqref{ppp} for $p=-1$, we come to the following result:
\begin{align*}
	\pa_tP_1&=\frac{\hh}{\hh^2-\rh^2}\,B,\\
	\pa_xP_1&=\frac{1}{2\rh}
	\begin{pmatrix}0&-1\\-1&0\end{pmatrix}
	+\frac{x}{2\rh(\hh^2-\rh^2)}\,B,\\
	\pa_yP_1&=\frac{1}{2\rh}
	\begin{pmatrix}0&-i\\ i&0\end{pmatrix}
	+\frac{y}{2\rh(\hh^2-\rh^2)}\,B,\\
	\pa_zP_1&=\frac{1}{2\rh}
	\begin{pmatrix}-1&0\\0&1\end{pmatrix}
	+\frac{z}{2\rh(\hh^2-\rh^2)}\,B.
\end{align*}
where
\[
B=\left(\!\!\begin{array}{cc}
	z&x+i y\\
	x-i y&-z
\end{array}\!\!\right).
\]

Now, we are able to compute the NC counterpart of the first Chern class. By definition,
\[
\operatorname{Ch}_1(P_1) := \Tr(P_1\,\dd P_1\,\dd P_1)
=\sum_{i<j}\Tr\Bigl(P_1\big(\partial_{a_i}P_1\partial_{a_j}P_1-\partial_{a_j}P_1\partial_{a_i}P_1\big)\Bigr)\,\dd {a_i}\wedge \dd {a_j},
\]
where $a_i\in\{t,x,y,z\}$. Matrix multiplication is performed in the written order.

Using the above matrices for partial derivatives of $P_1$ and the relation $x^2+y^2+z^2=\rh^2-\hh^2$ we get:
\begin{align}
	\operatorname{Ch}_1(P_1)
	&= \Tr(P_1\,\dd P_1\,\dd P_1) \notag\\
	&\quad-\frac{i\hh^2}{\rh^3(\hh^2-\rh^2)}
	\left(x\,\dd t\wedge \dd x+y\,\dd t\wedge \dd y+z\,\dd t\wedge \dd z\right) \notag\\
	&\quad-\frac{i\left(5\hh^3-2\hh^2\rh-2\hh\rh^2+\rh^3\right)}
	{2\rh^3(\hh-\rh)^2(\hh+\rh)} \notag\\
	&\qquad\times
	\left(z\,\dd x\wedge \dd y+y\,\dd z\wedge \dd x+x\,\dd y\wedge \dd z\right).
	\label{first-q-Ch}
\end{align}

Now, let us discuss the classical limit of the above quantum Chern class. Let $\hh\to 0$ and $\rh\to r$. Then the coefficients
of elements containing $\dd t$ vanish:
\[
-\frac{i\hh^2}{\rh^3(\hh^2-\rh^2)}\longrightarrow 0.
\]
For the terms without $\dd t$ the coefficient  has the following limit
\[
-\frac{i(5\hh^3-2\hh^2\rh-2\hh\rh^2+\rh^3)}
{2\rh^3(\hh-\rh)^2(\hh+\rh)}
\longrightarrow
-\frac{i}{2r^3}.
\]
Therefore the corresponding characteristic differential two-form has the classical limit
\[
\Tr(P_1\,\dd P_1\,\dd P_1)
\longrightarrow
-\frac{i}{2r^3}
\bigl(z\,\dd x\wedge \dd y+y\,\dd z\wedge \dd x+x\,\dd y\wedge \dd z\bigr).
\]

This coincides  (up to a normalizing factor) with the Chern form of  the classical rank-one projector on the usual sphere of radius
$r$. Equivalently, the normalized first Chern class on this sphere  is
\[
\operatorname{Ch}_1(P_1)
=\frac{1}{2\pi}\Tr(P_1\,\dd P_1\,\dd P_1)
\longrightarrow
-\frac{1}{4\pi r^3}
\bigl(z\,\dd x\wedge \dd y+y\,\dd z\wedge \dd x+x\,\dd y\wedge \dd z\bigr).
\]
On the oriented sphere $x^2+y^2+z^2=r^2$, this is, up to the chosen orientation/sign convention for
$P_1$, the usual area form divided by $4\pi r^2$. Note that its integral is $-1$ for the orientation induced by
$x\,\dd y\wedge \dd z+y\,\dd z\wedge \dd x+z\,\dd x\wedge \dd y$.

It should be emphasized that we are dealing with NC analogues of a four-dimensional vector space and the 2-form can be integrated
only if it is restricted onto a variety of dimension 2. However, the second Chern form of the idempotent under consideration is
\[
\Tr\,(P_1\,\dd P_1\,\dd P_1)^2.
\]
It is a four-form, and a direct ordered computation gives zero.

\begin{remark}
	In  numerous publications a so-called fuzzy sphere is considered. It can be defined as a quotient of the algebra $\AA$ \eqref{quot}
	over the ideal generated by $t$ and
	$\rh-c$, where $c=const>0$. The cotangent fibre on such a NC sphere can be introduced as the following quotient
	\[
	\Om(u(2)_h)\ot U(u(2)_h)/\langle \dd t,\,\, \dd x\, x+\dd y\, y+\dd z\, z \rangle.
	\]
\end{remark}

\section{Further generalizations}\label{sec:generalizations}

Note that the definition of the QPD on the algebra $\UU$ enables us to introduce the de Rham operator without using the classical
Leibniz rule. Moreover, our differential calculus developed on the algebra $\UUU$ is a deformation of that on the algebra $\Sym(u(2))$.

Consider now the algebras $\UU$ and the QPD introduced in section 2. It is convenient to introduce the shifted
derivatives $\tpa_i^j$ and the matrix $\tatt$ by the following rules:
\[
\tpa_i^j=\pa_i^j+\frac{1}{h}\,\de_i^j,\qquad \tatt=i\hh\|\tpa_i^j\|^t,
\]
where $t$ stands for  the transpose. The corresponding map is also multiplicative:
\begin{equation}
	\tatt(ab)= \tatt(a)\tatt(b),\quad \forall\, a,b \in \UU.
	\label{prope}
\end{equation}

The generating matrix $M$ of the algebra $\UU$ satisfies the Cayley-Hamilton identity \eqref{CH0} with central coefficients and,
therefore, one can introduce the corresponding eigenvalues $\mu_i$, $1\le i\le N$. Now, a natural question arises: how to extend
the QPD to these quantities. Under the requirement that the multiplicativity \eqref{prope} is preserved for the extended algebra we
arrive at the following system:
\[
\sum_i^N\, Q_i=a_1 I,\qquad  \sum_{i<j}^N\, Q_i\, Q_j=a_2 I,\quad \dots\quad  \prod_{i}^N\, Q_i=a_N I,
\]
where $Q_i:=\tatt(\mu_i)$. Also, note that the matrices $Q_i$ should commute with each other. So far, a solution to this system
has been found for $N=2$ only. However, once such matrices are found the characteristic forms can be computed as above:
\[
\Tr\,(P_i\, \dd P_i\, \dd P_i)^k.
\]

The next step of  generalization is related to the RE algebras $\MM(R)$.  Each such algebra is related to a
braiding $R$. We assume $R$ to be a skew-invertible Hecke symmetry.

\begin{remark}
	For the terminology and details on braidings $R$ the reader is referred to \cite{GPS1}. Here we only note that for any
	skew-invertible braiding $R\in \mathrm{End}(V^{\otimes 2})$, where $V$ is a  vector space of dimension $N<\infty$,
	there exists an extension to a braiding
	\[
	\mathcal{R}\in \mathrm{End}\big((V\oplus  V^*)^{\ot 2}\big),
	\]
	where $V^*$ is the dual space to $V$. For such a braiding there exists a skew-inverse operator $\Psi$ with the
	following property:
	\begin{equation}
		\Tr_{(2)}R_{12}\Psi_{23} = \Tr_{(2)}\Psi_{12}R_{23} = P_{13}.
		\label{def-psi}
	\end{equation}
	Here $R_{12} = R\otimes \mathrm{id}_V$ and $R_{23} = \mathrm{id}_V\otimes R$ are embeddings of $R$ into
	$\mathrm{End}(V^{\otimes 3})$.
	
	The partial traces of the operator $\Psi$ define the operators $B$ and $C$ by the rules:
	\[
	B_2 = \Tr_{(1)}\Psi_{12}, \quad C_1 = \Tr_{(2)}\Psi_{12}.
	\]
	The braiding $R$ is called strictly skew-invertible if $B$ and  $C$ are invertible (note that they are invertible simultaneously). If in addition $R$ is an even
	braiding with the birank $(m|0)$ (a positive integer $m\le N$) then
	\[
	BC = q^{-2m}I.
	\]
	
	These operators are very useful for describing the structure of the corresponding RE algebra. In particular, the operator $C$ enters
	the construction of the $R$-trace $\Tr_R X=\Tr (C X)$, where $X$ is any square $N\times N$ matrix. This $R$-trace replaces the
	usual trace in all constructions related to the RE algebras.
\end{remark}

However, in this case  we do not know any analogue of the matrix $\tatt$ with similar properties.
We can only define the QPD  via the so-called quantum doubles. As we described above,
each of them consists of two associative  unital algebras, endowed with some permutation relations.

The system introduced in \cite{GPS2} gives rise to a braided analogue of the Heisenberg--Weyl algebra:
\begin{align}
	R_{12}M_1R_{12}M_1-M_1R_{12}M_1R_{12}
	&=h(R_{12}M_1-M_1R_{12}), \notag\\
	R_{12}^{-1}D_1R_{12}^{-1}D_1-D_1R_{12}^{-1}D_1R_{12}^{-1}
	&=0, \label{q-Weyl}\\
	D_1R_{12}M_1R_{12}-R_{12}M_1R_{12}^{-1}D_1
	&=R_{12}+hD_1R_{12}. \notag
\end{align}

Here, $M_1 = M\otimes I$, $D_1= D\otimes I$, where $M=\|m_i^j\|$ is the generating matrix of the RE algebra, the matrix
$D=\|\pa_i^j\|$ is composed of the QPD. The permutation relations presented in the last matrix equality enable us to compute
the action of the QPD on any polynomial in the generators $m_i^j$. It is a quantum analog of the Leibniz rule for QPD.

Denote by $\MM(R)$ (respectively $\DD(R)$) the RE algebras generated by the matrix $M$ (respectively by $D$). If a Hecke
symmetry $R$ is even\footnote{If $R$ is not even, the Cayley-Hamilton identity also exists but it is more sophisticated.} and its
bi-rank is $(m|0)$, the  matrix $M$ satisfies the Cayley-Hamilton identity, similar to that \eqref{CH0}, but with $m$ instead of $N$.
It should be emphasized that the coefficients of the corresponding polynomial are always central. So, the central eigenvalues
$\mu_i$, $1\le i\le m$, of the matrix $M$ can be introduced in the same way as above. However, in this case we do not know
how to extend the action of the QPD onto these quantities.

Nevertheless, the de Rham operator and the corresponding differentials can be defined and now we pass to their construction.

First, compute the action of the QPD on the generators of the algebra $\MM(R)$. In accordance with the general construction
described above we apply the counit $\varepsilon(D) = 0$ to the rightmost derivatives $D$ in the last line of the
system  \eqref{q-Weyl} and find:
\begin{equation}
	D_1\triangleright R_{12} M_1R_{12} = R_{12} \quad  \Leftrightarrow \quad D_1\triangleright R_{12} M_1= I \quad
	\Leftrightarrow \quad D_1\triangleright M_{\ov 2}=R_{12}^{-1}.
	\label{middle}
\end{equation}
Here $\triangleright$ stands for the action of a linear operator and $M_{\overline 2} = R_{12}M_1R_{12}^{-1}$.

With the use of \eqref{def-psi} one can get an explicit formula for the action of a given $\partial_i^{\,j}$ on any generator $m_a^b$.
Indeed, upon multiplying the middle equality of \eqref{middle} from the right hand side by $\Psi_{23}$
and then applying the usual trace in the second position, we obtain, by \eqref{def-psi}:
\[
D_1 \triangleright M_2=P_{12}\, B_1 \quad \Leftrightarrow \quad \pa_i^j\triangleright m_a^b=\de_i^b B_a^j.
\]

The permutation relations in the third line of the system \eqref{q-Weyl} allow one to compute the action of $D$ on elements of
degree two:
\begin{equation}
	D_1\triangleright M_{\ov 2}\, M_{\ov 3}=R_{12}^{-1} M_{\ov 3}+M_{\ov 2}R_{12}^{-1} R_{23}^{-1}R_{12}^{-1}+
	h R_{12}^{-1}R_{23}^{-1}.
	\label{secdeg}
\end{equation}
Here $M_{\overline 3} = R_{23}M_{\overline 2}R_{23}^{-1}.$

Let us introduce new elements $d_i^{\,j}$, $1\leq i,j \leq N$, (their meaning and properties will be explained below) and compose
the matrix $\DDDD=\|d_i^{\,j}\|$. The de Rham operator is defined as follows:
\begin{equation}
	\dd=\Tr_R(\DDDD D):=\sum_{i,j,k}^N C_i^j \, d_j^{\,k}\, \pa_k^{\,i}.
	\label{ddR}
\end{equation}

The meaning of the  elements $d_i^{\,j}$ becomes clear if we apply the de Rham operator $\dd$ to the elements $m_k^l$:
\begin{equation}
	\dd(m_i^j)=C_a^b\, d_b^{\,c}\, \pa_c^{\,a}(m_i^j)= C_a^b\, d_b^{\,c}\, B_i^a \, \de _c^j= (BC)_i^b\, d_b^{\,j}=q^{-2m} d_i^{\,j}.
	\label{diff}
\end{equation}
Thus, up to a numerical factor the element $d_i^{\,j}$ equals the result of  applying the de Rham operator to the generator $m_i^j$.
So, we call the elements $d_i^{\,j}$ the {\em pure differentials}.

In the formulae below, $\dd M$ denotes the matrix composed of the elements $\dd(m_i^j)$; it differs from the matrix
$\DDDD=\|d_i^{\,j}\|$ by the scalar factor $q^{-2m}$.

\begin{remark}
	The above constructions are still valid for the RE algebra related to a Hecke symmetry of general bi-rank $(m|n)$. In this case the
	factor $q^{-2m}$ should be replaced by $q^{-2(m-n)}$. By setting $q=1$ we get the de Rham operator on the super-algebra
	$U(gl(m|n)_h)$.
\end{remark}

Now, we assume that the generators $d_i^{\,j}$ are skew-commutative in the following sense:
\begin{equation}
	R_{12}\DDDD_1R_{12} \DDDD_{1}=-\DDDD_1R_{12} \DDDD_{1} R_{12}^{-1}.
	\label{skew}
\end{equation}
Denote the algebra generated by differentials $d_i^{\,j}$ subject to \eqref{skew} by $\Dif$.
In addition, we define the permutation relations between elements $m_i^j$ and $d_i^{\,j}$ by the rule:
\begin{equation}
	R_{12}M_1R_{12} \DDDD_{1} = \DDDD_1 R_{12} M_{1} R_{12}.
	\label{perm1}
\end{equation}
We need these permutation relations in order to bring any differential form to the canonical presentation. This means that all differentials
are put on the leftmost position. However, as we said above, this canonical form can be only used on the last step.

Observe that the defining relations for the algebra $\DD(R)$ (the second line in \eqref{q-Weyl}) can be written in the following equivalent
forms:
\[
R_{12} D_{\ov 2} D_1=D_{\ov 2} D_1 R_{12}, \quad \Leftrightarrow \quad R_{12} D_{1} D_{\un 2}=D_{1} D_{\un 2} R_{12},
\]
where $D_{\un 2}=R_{12}^{-1} D_1R_{12}$.

Consider now the operator $\tQ$ acting in the vector space\footnote{Here the matrix elements $\partial_i^{\,j}$ of $D$ are considered
	to be the \emph{free} generators that are not subject to any permutation relations.} $\operatorname{span}(D_{\ov 2}D_1)$ as follows:
\[
\tQ(D_{\ov 2}\, D_1)=R_{12}(D_{\ov 2} D_1) R_{12}^{-1}.
\]
Note that this operator slightly differs from the operator $Q$ defined in \cite{GPS1}, Section 5.

If $R$ is a Hecke symmetry, then the operator $\tQ$ obeys the third-degree identity:
\[
(\tQ+q^2\, I)(\tQ+q^{-2}\, I)(\tQ-I)=0.
\]
Thus, the operator $\tQ$ has three eigenvalues $1$, $-q^2$ and $-q^{-2}$. By means of this operator we
construct the idempotents $\SSS$ and $\AAA$ (see \cite{GPS1} for details):
\[
\mathcal{S} = \frac{1}{(q+q^{-1})^2}\left( (q^2+q^{-2})I +\tQ+\tQ^{-1}\right), \quad
\mathcal{A} = \frac{1}{(q+q^{-1})^2}\left( 2I - \tQ - \tQ^{-1}\right).
\]
Using the above identity on $\tQ$ one can prove that $\mathcal{S}$ and $\mathcal{A}$ are indeed idempotents and, moreover, they
are complementary in the usual sense:
\[
\mathcal{S}^2 = \mathcal{S},\qquad \mathcal{A}^2 = \mathcal{A}, \qquad
{\SSS\, \AAA=\AAA\,\SSS}=0,\qquad {\SSS+\AAA}=I.
\]

\begin{lemma}\label{lem:8}
	
	\begin{enumerate}
		\item The defining relations \eqref{q-Weyl} of the algebra $\DD(R)$ can be written in the following equivalent form:
		\[
		R_{12}^{-1} D_1 R_{12}^{-1}  D_1 = D_1 R_{12}^{-1} D_1 R_{12}^{-1}\quad \Leftrightarrow\quad
		D_{\ov 2} D_1= {\SSS}(D_{\ov 2} D_1).
		\]
		\item
		The defining relations \eqref{skew} of the algebra $\Dif$ can be written in the following equivalent form:
		\[
		R_{12}\DDDD_1R_{12} \DDDD_{1}=-\DDDD_1R_{12} \DDDD_{1} R_{12}^{-1} \quad \Leftrightarrow\quad
		\DDDD_1 \DDDD_{\ov 2}={\AAA}(\DDDD_1 \DDDD_{\ov 2}).
		\]
	\end{enumerate}
\end{lemma}

Lemma \ref{lem:8} enables us  to prove the basic property of the de Rham operator.

\begin{proposition}
	The de Rham operator \eqref{ddR} satisfies $\dd^2=0$.
\end{proposition}

\begin{proof}
	By definition of the multiple action of $\dd$, we have:
	\[
	\dd^2= \Tr_{R(12)}\, (\DDDD_1\, \DDDD_{\ov 2}\, D_{\ov 2}\, D_1) = \Tr_{R(12)}(X_{12}Y_{12}),
	\]
	where we introduced the notation $X_{12}=\DDDD_1\, \DDDD_{\ov 2}$ , $Y_{12}=D_{\ov 2}\, D_1$
	to simplify writing.
	
	Then, due to Lemma \ref{lem:8} we get
	\[X_{12} =  {\AAA}(X_{12})=\frac{1}{(q+\qq)^2}\left(2 X_{12}-\tQ(X_{12})-\tQ^{-1}(X_{12})\right).
	\]
	On substituting this in $\dd^2$ we come to the result:
	\begin{align*}
		\Tr_{R(12)}\, \DDDD_1\, \DDDD_{\ov 2}\, D_{\ov 2}\, D_1
		&=\Tr_{R(12)}\,X_{12}\, Y_{12} \\
		&=\frac{1}{(q+\qq)^2}
		\Tr_{R(12)}\,(2\,X_{12}-R\, X_{12}\, R^{-1}-R^{-1}\,X_{12}\, R)\, Y_{12} \\
		&=\frac{1}{(q+\qq)^2}
		\Tr_{R(12)}\,(2\,X_{12}\, Y_{12}- X_{12}\, R^{-1}\, Y_{12}\, R
		-X_{12}\, R\, Y_{12}\, R^{-1}).
	\end{align*}
	Here, in order to pass to the third line we employed the cyclic property of the $R$-trace. Now, on taking into account that in the
	algebra $\DD(R)$ the following holds true
	\[
	R^{\pm 1}\, Y_{12}\, R^{\mp 1}= Y_{12},
	\]
	we conclude that the above expression for $\dd^2$ vanishes.
\end{proof}

Now, we compute the action of the de Rham operator $\dd$ on the power sums $p_k(M)=\Tr_R M^k$ for $k=1$ and $k=2$.
For $k=1$ the result follows immediately from \eqref{diff}:
\[
\dd p_1 =q^{-2m}\Tr_R\DDDD.
\]
For $k=2$ we obtain:
\begin{align*}
	D_1\triangleright p_2
	&= \Tr_{R(23)}\!\left(D_1\triangleright
	(M_{\ov 2}M_{\ov 3}R_{23})\right)\\
	&=2q^{-2m}M_1-q^{-2m}(q-\qq)(\Tr_R M)I_1
	+h q^{-3m}m_q I_1.
\end{align*}
Here, we used the definition of the second power sum, formula \eqref{secdeg}, and the following properties  of $R$-matrix calculus
(see \cite{GPS1}):
\[
\Tr_{R(2)} R_{12}^{-1} = q^{-2m} I_1,\qquad  \Tr_R I=q^{-m} m_q.
\]
Consequently, we have
\[\dd p_2 =2q^{-2m}\Tr_{R} (\DDDD M)-q^{-2m} (q-\qq) (\Tr_R \DDDD)\, (\Tr_R M)+h\, q^{-3m} m_q \Tr_{R} \DDDD.
\]

As for the Chern differential forms of projective modules corresponding to idempotents $P_i$, they are defined by the formula
\begin{equation}
	\operatorname{Ch}^R_k(P_i)=\Tr_R\,(P_i\,\dd P_i\,\dd P_i)^k.
	\label{Chern}
\end{equation}
As usual, in all constructions related to the RE algebras the classical trace is replaced by its $R$-analog.

Let us consider an example related to the well-known Drinfeld-Jimbo Hecke symmetry in the case $N=2$. In a suitable basis,
the symmetry is given by the following matrix:
\begin{equation}
	\left(\begin{array}{lcll}
		q&0&0&0\\
		0&\,\,q-\qq&1&0\\
		0&1&0&0\\
		0&0&0&q
	\end{array} \right).
\end{equation}
In this case  $m=N=2$.

Consider the  modified RE algebra corresponding to this Hecke symmetry. The generators of the algebra are matrix
elements of the generating matrix $M$:
\begin{equation}
	M = \left(\begin{array}{ll}
		m_1^1&m_1^2\\
		m_2^1& m_2^2 \end{array} \right)=\left(\begin{array}{ll}
		a&b\\
		c&d
	\end{array} \right).
	\label{matt}
\end{equation}
The permutation relations \eqref{q-Weyl} with the above Drinfeld-Jimbo symmetry $R$ give rise to the following explicit formulae:
\[
\begin{array}{lcl}
	q ab-q^{-1}\, ba=hb &\qquad& q (bc-cb)=((q-\qq)\,a-h)(d-a)\\
	\rule{0pt}{5mm}
	q ca-q^{-1}ac=h c&& q (db-bd)=((q-\qq)\, a -h)\, b\\
	\rule{0pt}{5mm}
	ad-da=0&&q(cd-dc)=c((q-\qq)\, a-h).
\end{array}
\]
The RE algebra generated by the elements $a$, $b$, $c$ and $d$ will be denoted $\Aqh$.

The generating matrix $M$ obeys the Cayley-Hamilton identity:
\[
M^2-\bigl(qp_1+q^{-1}h\bigr)M +\frac{q^2}{2_q} \left(\bigl(qp_1+q^{-1}h\bigr)p_1-p_2\right)I=0,
\]
where we denote $2_q = q+q^{-1}$. The polynomials $p_1 = \Tr_R(M)$ and $p_2 = \Tr_R(M^2)$ are the first two power sums.

Thus, it is natural to define the eigenvalues $\mu_1$ and $\mu_2$ of the matrix $M$ as solutions of the system of equations:
\begin{equation}
	\mu_1+\mu_2=qp_1+q^{-1}h,\qquad \mu_1\mu_2 = \frac{q^2}{2_q}
	\left(\bigl(qp_1+q^{-1}h\bigr)p_1-p_2\right).
	\label{def-mu}
\end{equation}
The eigenvalues $\mu_i$ are central elements of the extended algebra $\Aqh[\mu_1,\mu_2]$.

The QPD $\pa_i^{\,j}$ are defined in the algebra $\Aqh$ according to the permutation relations exhibited in \eqref{q-Weyl}.
However, we do not know any rule for their actions in the extended algebra.
Nevertheless, we can proceed further if we assume the eigenvalues $\mu_i$ to be constant.
Otherwise stated, we consider the following quotient of the algebra $\Aqh$:
\[
{\Aqh}/\langle \Tr_R M-\al_1,\,\Tr_R M^2-\al_2\rangle,
\]
where $\al_1$ and $\al_2$ depend on $\mu_1$ and $\mu_2$ in accordance with \eqref{def-mu}:
\[
\al_1=q^{-1}(\mu_1+\mu_2)-q^{-2}h,
\qquad
\al_2=(\mu_1+\mu_2)\al_1-q^{-2}2_q\,\mu_1\mu_2.
\]

This quotient is a two-parameter quantum hyperboloid. The corresponding differential algebra is generated by the elements
$d_i^j\ot u$, $u\in \Aqh$ modulo the differential ideal generated by the elements $\dd p_1=\dd(\Tr_R M)$ and
$\dd p_2 = \dd(\Tr_R M^2)$. Recall that the generators $d_i^j$ are skew-commutative in the sense of the relation \eqref{skew}.

As for the first Chern form in this example, it is defined by formula \eqref{Chern} with $k=1$ modulo the mentioned ideal. In this
computation it is necessary to take into consideration that, in the quotient just described, the central spectral parameters are fixed.
Since
\[
P_1=\frac{M-\mu_2 I}{\mu_1-\mu_2},
\]
then
\[
\dd P_1=\frac{\dd M}{\mu_1-\mu_2}
\]
and the first Chern differential form can equivalently be written as
\[
\Ch^R_1(P_1)=\Tr_R(P_1\,\dd P_1\,\dd P_1)
=\frac{1}{(\mu_1-\mu_2)^3}
\Tr_R\,\bigl((M-\mu_2 I)\,\dd M\,\dd M\bigr).
\]

Let us clarify  that the matrix $\dd M\,\dd M$ should be treated as follows. Each entry of this matrix is a product of
skew-commutative factors. Thus, we can apply the skew-symmetrizer $\AAA$  to each entry. Also, by employing the
permutation relation
\[
R_{12}M_1R_{12}\dd M_1=\dd M_1 R_{12}M_1R_{12},
\]
which is equivalent to \eqref{perm1}, we can present the Chern class $\Ch^R_1(P_1)$ in a canonical form, that is,
draw all differentials to the leftmost position.

It should be emphasized that  both operations commute with the action of $U_q(sl(N))$, provided the Hecke symmetry comes
from this quantum group. In general, this condition should be expressed in terms of the coaction of the RTT algebra.
In a similar manner we can define the higher Chern classes.

Thus, under the assumption that $R$ stems from $U_q(sl(N))$, we get the Chern classes depending on two parameters $q$ and $h$.
Their specializations may be summarized by the following commutative diagram of the Chern forms:
\[
\begin{array}{ccc}
	\operatorname{Ch}^R_1(P_1;h,q) & \xrightarrow{\quad q\to1\quad} &
	\operatorname{Ch}_1(P_1;h,1) \\
	\bigg\downarrow {\scriptstyle h\to 0} && \bigg\downarrow {\scriptstyle h\to 0} \\
	\operatorname{Ch}^R_1(P_1;0,q) & \xrightarrow{\quad q\to 1\quad} &
	\operatorname{Ch}^R_1(P_1;0,1)
\end{array}
\]
and similarly for other Chern classes.

The last element in this commutative diagram is a non-compact version of the classical form exhibited at the end of Section 5.
Here we are dealing with  a non-compact version, since in the RE algebra the compact form (the so-called quantum sphere) is not a
well-defined object. The point is that the quantum sphere should correspond to the conjugation operator, which  in turn
should correspond to a pairing in the basic space $V$. In the two-dimensional case $V=\operatorname{span}(x_1,x_2)$ (we consider the basis
in which  the Hecke symmetry $R$ is represented by the above matrix), the only categorical morphism (up to a numerical factor)
is a $q$-deformation of that
\[
\langle x_1,x_1 \rangle =\langle x_2,x_2 \rangle=0,\quad \langle x_1, x_2  \rangle=1,\quad \langle x_2,\, x_1  \rangle=-1.
\]
There is no categorical morphism which would be a deformation of the standard Euclidean or Hermitian form.

Note that by a categorical  morphism  we mean a map of an object of the monoidal quasitensor category generated by
the space $V$ to another object  which commutes with braidings in this category. Let us observe that for introducing
a conjugation in the RE algebra, which is nothing but the enveloping algebra of the space $\End(V)$, we need neither
conjugation in the quantum group nor that in its dual algebra.

\begin{remark}
	The quantum hyperboloid is a particular example of the so-called quantum varieties. The best and simplest way to
	construct them consists of quotienting a given RE algebra (modified or not) by the ideal generated by some
	combinations  of the central elements  $\Tr_R\, M^k$. In particular, in this way it is possible to construct quantum
	(or braided) analogs of semisimple orbits, that is, orbits of diagonal matrices in $gl(N)^*$ corresponding to the coadjoint
	action of $GL(N)$. The semiclassical counterparts of such quantum orbits can be obtained in two ways.
	
	The first way consists in reducing the Sklyanin bracket from the corresponding group $GL(N)$.  As was shown in \cite{KRR}
	the reduced bracket (often called the Poisson-Bruhat structure) is compatible with the Kirillov-Kostant-Souriau bracket
	if and only if such an orbit is symmetric, i.e. the matrix in question has only two distinct eigenvalues.
	
	The second way consists in restricting the Poisson bracket $\{\,,\,\}_{RE}$
	corresponding to the RE algebra to the orbit of any diagonal matrix. Moreover, the restricted bracket is compatible with
	the Kirillov-Kostant-Souriau one. This follows from the fact that the  bracket $\{\,,\,\}_{RE}$  is compatible with the linear $gl(N)$
	Poisson-Lie bracket. All these properties can be deduced from the paper \cite{Don}.
\end{remark}

In conclusion, let us briefly summarize. We developed some elements of differential  geometry on a specific family of NC algebras,
possessing very remarkable properties (see Introduction). In dealing with RE algebras, we  use the so-called $R$-traces which are
well-adapted to the structures of these algebras. In this connection we would like to pose the following problem: to describe a more
general class of associative algebras which require using generalized traces different from the usual ones or super-traces.

\bigskip

\noindent
{\bf Acknowledgements. }
V.~Roubtsov is grateful to the Institut des Hautes \'Etudes Scientifiques (IHES) for its hospitality and excellent working conditions.

\bigskip

\noindent
{\bf Funding.}
The work of P.~Saponov was carried out within the framework of the Basic Research Program at HSE University (HSE-BR-2025-84).

\end{document}